\documentclass[12pt,reqno]{amsart}
  
 \usepackage{comment}
\usepackage{tikz}
\usepackage{mathrsfs}
\usepackage{latexsym}
\usepackage{graphicx}
\usepackage{amssymb}
 \graphicspath{{Figures/}}
 \usepackage[show]{ed}

\renewcommand{\epsilon}{\varepsilon}

\newcommand{\R}{{\mathbb R}}

\newcommand{\Z}{{\mathbb Z}}

\newcommand{\supp}{{\operatorname{supp\,}}}
\renewcommand{\phi}{\varphi}

\newcommand{\ep}{\varepsilon}

\newtheorem{theo}{{\sc Theorem}}
\newtheorem{maintheo}{{\sc Theorem}}
\newtheorem{cor}{{\sc Corollary}}[section]

\newtheorem{lem}[cor]{{\sc Lemma}}

\newtheorem{prop}[cor]{{\sc Proposition}}
\numberwithin{equation}{section}

\newenvironment{rem}{\medskip\noindent{\it Remark:\/} }{\medskip}

\newtheorem{defn}[theo]{{\sc Definition}}

\usepackage{hyperref}

\title[Trace estimates and improved pointwise bounds]{Trace estimates and improved pointwise bounds for joint Eigenfunctions}

\author{Xianchao Wu and Xiao Xiao}
\address{School of Mathematics and Statistics, Wuhan University of Technology, Wuhan, Hubei, China}
\email{xianchao.wu@whut.edu.cn } 
\address{Department of Mathematics and Statistics, McGill University, Montr\'eal, Quebec, Canada}
\email{xiao.xiao5@mail.mcgill.ca}

\date{}

\begin{document}

\begin{abstract}
For $L^2$-normalized joint eigenfunctions in a quantum integrable system, \cite{GT2020} gave polynomial improvements over the standard H\"omander bounds for typical points. In this paper, we improve their result by establishing a sharp bound of $h^{\frac{-n+k+1}2}$ for the points satisfying a {\em rank $k$} non-degeneracy condition. 
\end{abstract}

\maketitle

\section{Introduction}\label{section1}

Let $(M,g)$ be a compact Riemannian manifold without boundary. By basic functional analysis, the Laplace-Beltrami operator $\Delta_g:=\textit{div}_g\circ\nabla_g$ has discrete spectrum $\{\lambda_j\}_{j=0,1,...}$. The corresponding eigenfunctions are $u_j$ satisfying
$$
-\Delta_gu_j=\lambda_ju_j.
$$
We opt to use the semiclassical notation. Let $h_j:=\lambda_j^{-\frac{1}{2}}$, and omitting the subscript $j$, the above equation becomes
\begin{equation}\label{1}
    (-h^2\Delta_g-1)u_h=0.
\end{equation}
The principal symbol of the operator $P(h)=-h^2\Delta_g-1$ is 
\begin{equation}\label{2}
p(x,\xi)=|\xi|_{g(x)}^2-1.    
\end{equation}
The symbol $p(x,\xi)$ is a smooth function on $T^*M$ and can be viewed as a Hamiltonian. The associated Hamiltonian flow is exactly the lifted geodesic flow on $T^*M$. (Under the standard projection, this flow is the geodesic flow on $M$.) By the quantum-classical correspondence, (\ref{1}) and (\ref{2}) should influence each other in the following sense: The dynamical properties of the geodesic flow (\ref{2}) of $(M,g)$ influence the asymptotic properties of $u_h$ in (\ref{1}) as $h\to 0^+$.

There has been an perennial endeavor to understand the concentration of eigenfunctions. A compelling question is how spiky can a sequence of eigenfunction be: Seeking for an $L^{\infty}$ bound of $L^2$ normalized $u_h$ in terms of $h$. 

Beginning in the 1950's, in a series of works by Levitan \cite{Lev}, Avakumovi\'c \cite{Ava}, and H\"ormander \cite{Hor}, the following estimate is obtained (known as the H\"ormander bound):
\begin{equation}\label{known}
\|u_h\|_{L^\infty}=O(h^{\frac{1-n}2}),
\end{equation}
as $h\to 0$ and \eqref{known} is saturated on the round sphere. This bound was improved to $o(h^{\frac{1-n}2})$ by Safarov, Sogge, Toth, Zelditch and Galkowski \cite{Saf88, SZ02, STZ11, SZ16a, SZ16b, GT17, Gal19} under various dynamical assumptions at $x$. When $(M,g)$ has no conjugate points, a quantitative improvement
\begin{equation}\label{know1}
\|u_h\|_{L^\infty}=O(h^{\frac{1-n}2}/\sqrt{|\ln h|}),
\end{equation}
as $h\to 0$, has been known since the classical work of B\'erard \cite{Ber77, Bon17, Ran78}. In recent times, Canzani and Galkowski \cite{CG21, CG23} developed the tool of geodesic beams to study the quantitative improvements without global geometric assumptions on $(M,g)$. 

\subsection{Quantum completely integrable (QCI) system}
A semiclassical pseudodifferential operator $P_1(h)$ is called quantum completely integrable (QCI) if there exist functionally independent semiclassical pseudodifferential operators $P_2(h),\cdots,P_n(h)$ such that 
$$
[P_i(h),P_j(h)]=0,\quad i,j=1,\cdots, n.
$$
Given real numbers $E_1(h),\cdots, E_n(h)$, we want to study the $L^2$ normalised functions $u_h$ that solves $P_j(h)u_h=E_j(h)u_h$ for all $j=1,\cdots,n$. We call them joint eigenfunctions corresponding to the joint eigenvalues $E(h):=(E_1(h),\cdots,E_n(h))\in \R^n$ of the system $\hat{\mathcal{P}}:=(P_1(h),\cdots, P_n(h))$.

The principal symbols $p_j\in C^{\infty}(T^*M)$ of each operator $P_j(h)$ give rise to a moment map that defines an associated classical integrable system. The moment map is denoted by
$$
\mathcal{P}:=(p_1,\cdots,p_n):T^*M\to \R^n.
$$

The regular set of $\mathcal{P}$, denoted by $(T^*M)_{reg}$, consists of points $(x,\xi)\in T^*M$ such that
$$
\dim span\{ dp_1(x,\xi),\cdots,dp_n(x,\xi)\}=n.
$$
As a global assumption, the classical system $\mathcal{P}$ is {\textit{Liouville integrable}}, i.e. the set $(T^*M)_{reg}$ is an open dense subset of $T^*M$.

The above construction induces a Lagrangian torus fibration of $T^*M$. If $E=(E_1,\cdots,E_n)$ is a regular value of $\mathcal{P}$, then by the Arnold-Liouville theorem, $\mathcal{P}^{-1}(E)$ consists of a finite disjoint union of Lagrangian tori:
$$
\mathcal{P}^{-1}(E)=\bigcup_{j=1}^N \Lambda_j(E).
$$
Each $\Lambda_j(E)$ is a Lagrangian submanifold of $T^*M$ and is fixed by the torus action of $\mathcal{P}$. In fact, the joint flow 
$$
G_{t}(x,\xi):=\exp (t_1H_{p_1})\circ \cdots\circ\exp(t_nH_{p_n})(x,\xi)
$$
is defined for all $t\in \R^n$ and the Hamiltonian flow $\exp (t_jH_{p_j})$ preserves the level set $\{p_j=E_j\}$. Furthermore, the definition does not depend on the order of composition since we assumed $[P_i,P_j]=0$, which implies $\{p_i,p_j\}=0$.

The singular set $T^*M\setminus(T^*M)_{reg}$ can be view as the degenerate Lagrangian tori.

Galkowski-Toth \cite{GT2020} investigated the pointwise bounds of $u_h$ for the case $P_1(h)$ being a Schr\"odinger operator. More explicitly, when $E_1$ is a regular value of $p_1$ and the system is of {\em Morse type} at point $x$, they established the following bound:
\begin{equation}\label{known2}
|u_h(x)|=O(I_n(h)),
\end{equation}
here $I_2(h)=h^{-\frac{1}4}$, $I_3(h)=h^{-\frac{1}2}|\ln h|^{\frac{1}2}$ and $I_{n\geq 4}(h)=h^{\frac{2-n}2}$.

\subsection{Main results}
One purpose of this paper is to obtain a little $o$ improvement of the pointwise bounds \eqref{known2} under certain non-degeneracy condition. We achieve this in Theorem \ref{main1}. A useful tool is developed along the way (Theorem \ref{trace}), which we call a {\textit{quantitative trace estimate}}. It is of independent interest. 

We say that $p_1$ is of real principal type on the hypersurface $p_1^{-1}(E_1)$: if $E_1>0$ is a regular value of $p_1$, for any $(x,\xi)\in p_1^{-1}(E_1)$, the following inequality holds,
\begin{equation}\label{principal}
\partial_\xi p_1(x,\xi)\neq 0.
\end{equation}
One sets
\begin{equation}\label{loop set}
\mathcal{C}_x:=\{\xi\in T_x^*M:\, p_1(x,\xi)=E_1\}.
\end{equation} 
and define the {\textit{loop set}} to be
\begin{equation}
    \mathcal{L}_x:=\{\xi\in \mathcal{C}_x:\exists t<+\infty,\, \Phi_{t,p_1}(x,\xi)\in \mathcal{C}_x\},
\end{equation}
where $\Phi_{t,p_1}$ is the Hamiltonian flow of $p_1$. If furthermore, a covector loops back in a direction infinitly close to it original direction infinitly many times in the future and the past, we call it a \textit{recurrent direction}. Define the \textit{recurrent set} to be

\begin{equation}\label{recurrent set}
\mathcal{R}_x:=\{\xi\in \mathcal{C}_x:\Big(\bigcap_{T>0}\overline{\bigcup_{t>T}\Phi_{t,p_1}(x,\xi)|_{\mathcal{C}_x}}\Big)\bigcap\Big(\bigcap_{T>0}\overline{\bigcup_{t>T}\Phi_{-t,p_1}(x,\xi)|_{\mathcal{C}_x}}\Big)\}.
\end{equation}
The closure is under the subspace topology from $\mathcal{C}_x$.

It is the omega-limit of the following dynamical system: Define the return-time function $T_x:\mathcal{C}_x\to \R\cup\{\infty\}$, 
$$
T_x(\xi)=\inf_{t>0}\{t:\Phi_{t,p_1}(x,\xi)\in \mathcal{C}_x\}.
$$
Then the return map $\widetilde\Phi_x:=\Phi_{T_x,p_1}:\mathcal{L}_x\to \mathcal{C}_x$ restricted to $\widetilde{\mathcal{L}}_x:=\bigcap_{k\in \Z}\widetilde{\Phi}^k(\mathcal{L}_x)$ is a dynamical system. By (\ref{recurrent set}), $\xi\in \mathcal{R}_x$ if and only if $\xi$ is a limit point of the orbit of $\widetilde\Phi_x(\xi)$. See \cite{STZ11} and \cite{Gal19}.

Note that we have the following inclusion:
$$
\mathcal{R}_x\subset\widetilde{\mathcal{L}}_x\subset\mathcal{L}_x\subset \mathcal{C}_x.
$$

\begin{defn}\label{def}
Suppose that $1\leq k\leq n-1$ and $\Gamma$ is a subset of $\mathcal{C}_x$. We say that $\mathcal{P}=(p_1, \dots, p_n)$ satisfies {\bf rank $k$ condition} at $x\in M$ for $\Gamma$ if
\begin{equation}\label{rk}
    \text{ for any } \xi\in\Gamma,\quad\text{dim span}\{\partial_\xi p_2 ,\,\partial_\xi p_3,\cdots, \partial_\xi p_n\} = k
\end{equation}
as a subspace in $T_{\xi}\mathcal{C}_x$, where $\partial_{\xi}$ denotes the full gradient along $\mathcal{C}_x$.

For brevity, we say $\mathcal{P}|_{\Gamma}$ is rank $k$ at $x$.
\end{defn}

\begin{rem}
    It is noteworthy that similar conditions were used in \cite[Def. 1.2]{EGK24}, under the name “fiber rank $k$ condition”, and in \cite[Thm 0.1 (2)]{Tacy19}, under the name “admissibility condition”. Our formulation is slightly more general in the sense that it specifies the energy level and the subset $\Gamma$ of energy surface of $p_1$. Geometrically, condition (\ref{rk}) means that $\mathcal{P}^{-1}(E)\cap \Gamma$ is a $k$ dimensional submanifold which can be parameterized by $(\xi_{i_1},\dots, \xi_{i_k})$.
\end{rem}

From now on, we always suppose that $\hat{\mathcal{P}}=\{\sqrt{-h^2\Delta},P_2(h),\dots, P_n(h)\}$ is a QCI system on a smooth, compact manifold without boundary, $(M,g)$. Let $\{u_h\}$ be the $L^2$-normalized joint eigenfunctions.
As a key ingredient of our paper, we can state our first theorem,
\begin{maintheo}\label{trace}
Assume $A(x, hD)\in \Psi^0(M)$ with symbol $a=\sum_{j=0}^\infty a_j h^j$. If $\mathcal{P}|_\Gamma$ is rank $k$ at $x$ with $\Gamma=\mathcal{C}_x\cap \supp a$, then there exists a positive constant $h_0$ so that for $h\in (0, h_0]$, 
\begin{equation}\label{main cor eq}
|A(x,hD) u_{h}|^2=O(h^{-n+k+1}).
\end{equation}
\end{maintheo}

\begin{rem}
Section \ref{section2} shows that Theorem \ref{trace} relies on the trace estimates, which also play a critical role in proving the local Weyl law (see \cite[Section 5]{So2014}).
\end{rem}




\begin{defn}\label{morse}
    We say that $\mathcal{P}=(p_1,\cdots,p_n)$ is {\bf Morse} at $x\in M$ if there exists a function $f\in C^{\infty}(\R^n,\R)$, such that the principal symbol $q$ of the operator $Q:=f(P_1,\cdots,P_n)$ satisfies 
    \begin{equation}
        q\big|_{\mathcal{C}_x}(\xi) \text{ is a Morse function in } \xi. 
    \end{equation}
\end{defn}

\begin{rem}
    For surfaces of revolution, let $\hat{\mathcal{P}}$ be the standard QCI system (see Section 4). Then $\mathcal{P}$ is \textit{not} Morse at north and south poles, because $P_2$ vanishes there and $f=f(P_1)$, the symbol of whom, restricted to the energy surface of $p_1$, is a constant function.
    
    Also for the triaxial ellipsoid, the standard QCI system is \textit{not} Morse at the four umbilical points \cite{GT2020}. 
\end{rem}

Our main theorem is 

\begin{maintheo}\label{main1}
If $\mathcal{P}|_{\Gamma}$ is rank $k$ \rm{($k\geq 2$ if $n>3$)} at $x$ with $\Gamma=\mathcal{R}_x\cap \bigcap\limits_{2\leq j \leq n}\{|p_j(x,\xi)-E_j|\leq \epsilon\} $ for some $\ep>0$, and $\mathcal{P}$ is Morse, then there exists a positive constant $h_0$ such that for any $h\in (0, h_0]$,
\begin{equation}\label{main1 eq}
| u_h(x) |  = o\big(I_n(h)\big),
\end{equation}
where $I_2(h)=h^{-1/4}$, $I_3(h)=h^{-1/2}|\log h|^{1/2}$ and $I_n(h)=h^{(2-n)/2}$ if $n>3$. 

Furthermore, if $\mathcal{P}|_{\Gamma'}$ is rank $k$ at $x$ with $\Gamma'=\mathcal{C}_x\cap \bigcap\limits_{2\leq j \leq n}\{|p_j(x,\xi)-E_j|\leq \epsilon\}$ for some $\ep>0$, then there exists a positive constant $h_0'$ such that for any $h\in (0, h_0']$,
\begin{equation}\label{main1 eq2}
| u_h(x) | = O(h^{\frac{-n+k+1}2}).
\end{equation}

\end{maintheo}

\begin{rem}
   \begin{enumerate}
       \item Regarding the first conclusion (\ref{main1 eq}): The pointwise bound for a general Laplace eigenfunction (without quantum integrability assumption) possess an improvement to little $o$ if $|\mathcal{L}_x|=0$, \cite{SZ02} (In fact, only need $|\mathcal{R}_x|=0$ ,\cite{STZ11}). It is tempting to ask if the same is true for joint quasimodes in the QCI case. The answer is negative. On a generic biaxial ellipsoid, if the point $x$ is on the equator, then $|\mathcal{R}_x|=|\mathcal{L}_x|=0$. However, one can construct a quasimode of magnitude $\sim  h^{-1/4}$ attained on the equator, see section \ref{section4}. That shows some extra condition such as the rank $k$ is necessary.
       \item Regarding the second conclusion (\ref{main1 eq2}): If the conditions there are satisfied for all points $x\in M$, we restore the $L^{\infty}$ bound in \cite{Tacy19}. We note that our approach is different with Tacy's. \cite{Tacy19} is based on the result of \cite{KTZ}, however, our approach is self-contained and works point-wise. Also, the condition rank $k$ arises naturally in our proof (ref. \eqref{phi low}, \eqref{rank2 1} and \eqref{rank2 2}).
   \end{enumerate}
\end{rem}

\subsection{Example and superintegrability}\label{section superint}



Let $M=\mathbb{T}^n$, $n\geq 3$, then $\mathcal{C}_x$ is the round cosphere. Let $p_2=\xi_2,\cdots,p_n=\xi_n$. On any pair of poles, at most one of $p_j$'s is degenerate. Therefore, $\mathcal{P}|_{\mathcal{C}_x}$ only satisfies a rank $n-2$ condition at any $ x\in \mathbb{T}^n$. As a consequence, (\ref{main1 eq2}) gives $|u_h(x)|=O(h^{-1/2})$, a bound that is not optimal. However, we can use a simple trick as follows:

Let $u_h$ be the joint eigenfunction of 
$$\hat{\mathcal{P}}=\{\sqrt{-h^2\Delta},-ih\partial_{x_2},-ih\partial_{x_3},\cdots,-ih\partial_{x_n}\}$$
with eigenvalue $E=(1,1,0,\cdots,0)$, then $\mathcal{P}$ degenerates at $(|\xi|,\xi_2,\cdots,\xi_n)=(1,1,0,\cdots ,0)$. But the same function $u_h$ can be viewed as a joint function 
$$
\hat{\mathcal{Q}}=\{\sqrt{-h^2\Delta},-ih\partial_{x_1},-ih\partial_{x_3},\cdots,-ih\partial_{x_n}\}
$$
with eigenvalue $F=(1,0,0,\cdots,0)$. $\mathcal{Q}$ satisfies the rank $n-1$ condition on $\Gamma'$. As a consequence, (\ref{main1 eq2}) gives $|u_h(x)|=O(1)$ which is sharp.

\begin{figure}[h]
\includegraphics[width=0.9\textwidth]{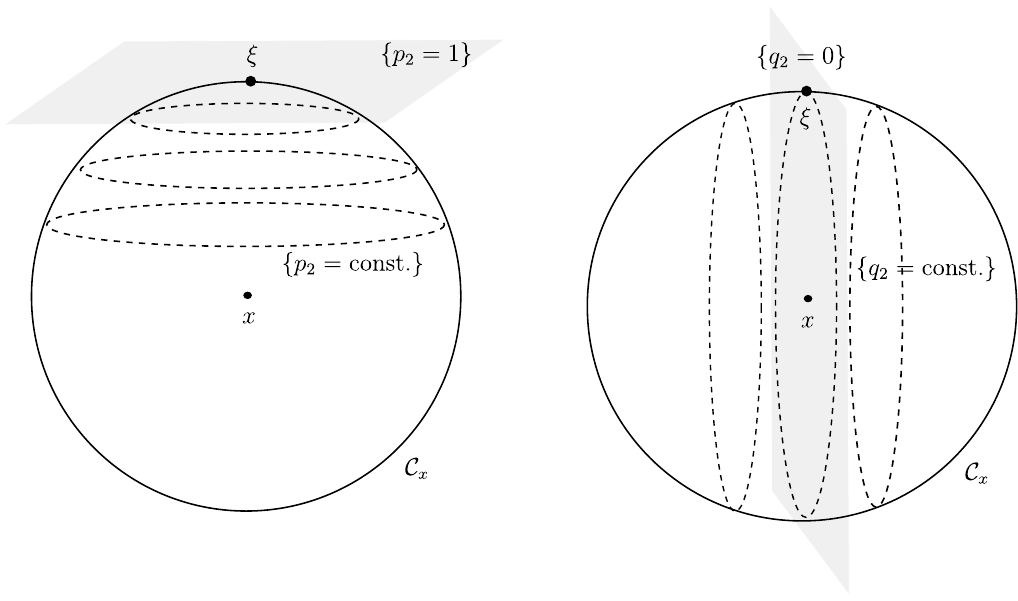}
\caption{On the left: At $\xi$, the system $\mathcal{P}$ only satisfies a rank $n-2$ condition; On the right: At $\xi$, the system $\mathcal{Q}$ satisfies the rank $n-1$ condition. The same function $u_h$ is a joint eigenfunction of both $\hat{\mathcal{P}}$ and $\hat{\mathcal{Q}}$.}
\label{Figure 1}
\end{figure}

This method is possible only if there are at least $n$ independent, commuting operators that also commute with the Laplacian. We will not delve deeper into this issue in this paper.

\subsection{Plan of the paper} Section \ref{section2} proves Theorem \ref{trace} by a Fourier inversion of the operator and computing the on-diagonal Schwartz kernel, which we call the trace estimate (\`a la Sogge). It is an important technical ingredient of this paper. In section \ref{section3}, we finish the proof of Theorem \ref{main1} by taking a closer look at the recurrent set and building test symbols accordingly. We conclude by a study of examples on surfaces of revolutions in section \ref{section4}.

\subsection*{Acknowledgments} We thank Weiwei Wang for valuable discussions at the beginning of preparing this paper.

\section{Trace estimates}\label{section2}

In this section, we aim to prove Theorem \ref{trace}. Firstly we establish following trace estimate which is similar to Proposition 2.2 in \cite{SZ02}, but in QCI case.

\begin{lem}\label{lwl}
Assume $A(x, hD)\in \Psi^0(M)$ with symbol $a=\sum_{j=0}^\infty a_j h^j$. Let $\rho\in C_0^\infty(\mathbb{R})$ vanish for $|t|\geq \epsilon_0$, $\epsilon_0>0$ small enough, and satisfy $\rho(0)=1$. Let $\Gamma:=\mathcal{C}_x\cap \supp a$. Then if $(p_1|_{\Gamma}, p_2|_{\Gamma})$ is {\bf rank $1$} at $x$, one has that
\begin{equation}\label{main lemma eq2}
\int_{\mathbb{R}} \hat{\rho}(t) \big(A e^{\frac{i}h t(P_1-E_1)}  \rho\big(h^{-1}[P_2(h)-E_2]\big)  A^*\big)(t,x,x) dt =O(h^{-n+2}).
\end{equation}

Otherwise, one has the following
\begin{equation}\label{main lemma eq}
\int_{\mathbb{R}} \hat{\rho}(t) \big(A e^{\frac{i}h t(P_1-E_1)}  \rho\big(h^{-1}[P_2(h)-E_2]\big)  A^*\big)(t,x,x) dt=O(1) I_h(x)^2,
\end{equation}
if $(p_1, p_2)$ is Morse type.
\end{lem}

\begin{proof}

We can write the Schwartz kernel of $e^{\frac{i}h t(P_1-E_1)}$ in the form 
\begin{equation}
(2\pi h)^{-n} \int_{\mathbb{R}^n}e^{\frac{i}h [\varphi_1(t, y_1,\xi_1)-\left<w_1, \xi_1\right>-tE_1]} a_1(t, y_1,w_1,\xi_1;h) d\xi_1+ O_{C^\infty}(h^\infty)
\end{equation}
where $a_1\sim \sum_{j=0}^\infty a_{1,j} h^j$, $a_{1,j}\in C^\infty$, $a_{1,0}\geq C>0$ and $\varphi_1(t, y_1,\xi_1)$ solves the eikonal equation
\begin{equation}\label{phi1}
\partial_t \varphi_1(t, y_1,\xi_1)=p_1(y_1, \partial_{y_1} \phi_1), \quad \varphi_1(0, y_1, \xi_1)=\left<y_1,\xi_1 \right>.
\end{equation}
And
\begin{equation}
\rho\big(h^{-1}[P_2(h)-E_2]\big)=\frac{1}{2\pi}\int_{\mathbb{R}}\hat{\rho}(s) e^{\frac{i}h s(P_2(h)-E_2)} ds.
\end{equation}
Here $e^{\frac{i}h s(P_2(h)-E_2)}$ has a Schwartz kernel that is of the form
\begin{equation}
(2\pi h)^{-n} \int_{\mathbb{R}^n}e^{\frac{i}h [\varphi_2(s, w_1,\xi_2)-\left<y_2, \xi_2\right>-sE_2]} a_2(s, w_1,y_2,\xi_2;h) d\xi_2+ O_{C^\infty}(h^\infty),
\end{equation}
where $a_2\sim \sum_{j=0}^\infty a_{2,j} h^j$, $a_{2,j}\in C^\infty$, $a_{2,0}\geq C>0$ and $\varphi_2(s,w_1,\xi_2)$ solves the eikonal equation
\begin{equation}\label{phi2}
\partial_s \varphi_2(s, w_1, \xi_2)=p_2(w_1, \partial_{w_1} \phi_2), \quad \varphi_2(0, w_1, \xi_2)=\left<w_1,\xi_2 \right>.
\end{equation}
From \eqref{phi1} and \eqref{phi2}, one can derive the following Taylor expansions for $\varphi_1(t, y_1,\xi_1)$ (resp. $\varphi_2(s, w_1,\xi_2)$) centered at $t=0$ (resp. $s=0$).
\begin{align}
\varphi_1(t, y_1,\xi_1)&=y_1\xi_1+ tp_1(y_1,\xi_1)+O_{y_1,\xi_1}(t^2), \\
\varphi_2(s, w_1,\xi_2)&=w_1\xi_2+ sp_2(w_1,\xi_2)+O_{w_1,\xi_2}(s^2).
\end{align}




In conclusion, we need to prove that
\begin{align}\label{main kernel}
h^{- 4n}\int &\hat{\rho}(s)\hat{\rho}(t) \exp[\frac{i}h \left<x-y_1, \xi\right>] \exp[\frac{i}h \Phi(s,t,y_1, w_1, y_2, \xi_1, \xi_2)]\exp[\frac{i}h \left<y_2-x, \eta\right>]\cdot  \nonumber \\
& c(s,t,y_1,w_1,y_2,\xi_1,\xi_2;h) a(x, \xi) \overline{a(x, \eta)} dy_1 d\xi dw_1 d\xi_1 d\xi_2 dy_2 d\eta dt ds,
\end{align}
can be bounded by the right hand side of \eqref{main lemma eq2} or \eqref{main lemma eq}, here $c(\cdot)=a_1(t, y_1,w_1,\xi_1;h) a_2(s,w_1,y_2,\xi_2;h) $ and
\begin{align*}
\Phi(s,t,y_1,w_1,y_2,\xi_1, \xi_2)=&\left<y_1-w_1,\xi_1 \right>+\left<w_1-y_2, \xi_2\right> +t\big(p_1(y_1,\xi_1)-E_1\big)+s\big( p_2(w_1,\xi_2)-E_2 \big)\nonumber \\
&+O_{y_1,\xi_1}(t^2)+O_{w_1,\xi_2}(s^2).
\end{align*}

Now we set
\begin{equation*}
c'(s,t,y_1,w_1,y_2,\xi_1,\xi_2;h)=\chi\big(p_1(y_1,\xi_1)-E_1\big) \chi\big(p_2(w_1,\xi_2)-E_2\big)c(s,t,y_1,w_1,y_2,\xi_1,\xi_2;h),
\end{equation*}
here $\chi\in C_0^\infty(\mathbb{R})$ vanish for $|t|\geq \epsilon$, $\epsilon>0$ small enough, and satisfy $\chi(0)=1$.
Hence \eqref{main kernel} is equal to
\begin{align}\label{main kernel1}
h^{- 4n}\int &\hat{\rho}(s) \hat{\rho}(t) \exp[\frac{i}h \left<x-y_1, \xi\right>] \exp[\frac{i}h \Phi(s,t,y_1,w_1,y_2,\xi_1,\xi_2)]\exp[\frac{i}h \left<y_2-x, \eta\right>]  \nonumber \\
 &c'(s,t,y_1,w_1,y_2,\xi_1,\xi_2;h) a(x, \xi) \overline{a(x, \eta)} dy_1 d\xi dw_1 d\xi_1 d\xi_2 dy_2 d\eta dt ds + O(h^\infty).
\end{align}

One can apply stationary phase in $(y_1, \xi)$-variables in \eqref{main kernel1}. The critical point equations for $(y_1, \xi)$ are
\begin{align}
\xi&=\xi_1+ t\partial_{y_1} p_1(y_1,\xi_1)+ O(t^2), \\
y_1&=x.
\end{align}

With applying Taylor expansion at the critical point $(y_{1c}, \xi_{c})$ one has that \eqref{main kernel1} equals
\begin{align}
I (x,x; h):=&h^{-3n}\int \hat{\rho}(s) \hat{\rho}(t)  \exp\Big[\frac{i}h\Big(\left<x-w_1, \xi_1\right>+\left<w_1-y_2,\xi_2\right>+ t(p_1(x,\xi_1)-E_1\big)+\nonumber \\
&s\big(p_2(w_1,\xi_2)-E_2\big) +\left<y_2-x, \eta\right>+O_{w_1,\xi_2}(s^2)+O_{x,\xi_1}(t^2) \Big)\Big]\cdot  \nonumber\\
& c_1(s,t,x, w_1,y_2,\xi_1,\xi_2;h)[a(x, \xi_1)+o(1)] \overline{a(x, \eta)} dw_1 d\xi_1 d\xi_2 dy_2 d\eta dt ds,   \label{main kernel2}
\end{align}
here $c_1\sim \sum_{j=0}^\infty c_{1,j} h^j$ and $c_{1,j}\in C^\infty$.

One can apply stationary phase in $(y_2, \eta)$-variables in \eqref{main kernel2}. The critical point equations for $(y_2, \eta)$ are
\begin{align}
\eta&=\xi_2, \\
y_2&=x.
\end{align}
With applying stationary phase at the critical point $(y_{2c}, \eta_{c})$ and Taylor expansion one has that \eqref{main kernel2} equals
\begin{align*}
I (x,x; h)=&h^{-2n}\int \hat{\rho}(s) \hat{\rho}(t)  \exp\Big[\frac{i}h\Big(\left<x-w_1, \xi_1-\xi_2\right>+ t(p_1(x,\xi_1)-E_1\big)+s\big(p_2(w_1,\xi_2)-E_2\big)  \nonumber \\
+&O_{w_1,\xi_2}(s^2)+O_{x,\xi_1}(t^2) \Big)\Big]  c_2(s,t,x, w_1,\xi_1,\xi_2;h) [a(x, \xi_1)+o(1)] \overline{a(x, \xi_2)}  dw_1 d\xi_1 d\xi_2 dt ds,
\end{align*}
here $c_2\sim \sum_{j=0}^\infty c_{2,j} h^j$ and $c_{2,j}\in C^\infty$.

Next one would like to apply stationary phase in $(w_1, \xi_2)$-variables. The critical point equations for $(w_1, \xi_2)$ are
\begin{align}
\xi_2&=\xi_1-s\partial_{w_1} p_2(w_1,\xi_2)+ O(s^2), \\
w_1&=x-s\partial_{\xi_2} p_2(w_1,\xi_2)+ O(s^2).
\end{align}

Hence with applying Taylor expansion one can get that
\begin{align*}
I (x,x; h)=&h^{-n}\int \hat{\rho}(s) \hat{\rho}(t)  \exp\Big[\frac{i}h\Big(  t(p_1(x,\xi_1)-E_1\big)+s\big(p_2(x,\xi_1)-E_2\big)+O_{x,\xi_1}(t^2)  \nonumber \\
&+O_{x,\xi_1}(s^2) \Big)\Big]  c_3(s,t,x,\xi_1;h) [a(x, \xi_1)+o(1)] [\overline{a(x, \xi_1)}+o(1)]  d\xi_1 dt ds.
\end{align*}
here $c_3\sim \sum_{j=0}^\infty c_{3,j} h^j$ and $c_{3,j}\in C^\infty$.

Notice that $P_1(h)=-h^2\Delta_g$ with $E_1=1$. We use the geodesic normal coordinate about $x$ and can make the change of variables $\xi_1=r\Theta$, where $r=|\xi_1|$ and $\Theta=(\theta_1,\theta_2,\dots,\theta_{n-1})\in S^{n-1}$. One can get
\begin{align}\label{tilde I}
I &(x, x; h)= h^{-n}\int\int\int\int \hat{\rho}(s) \hat{\rho}(t)  \exp\Big[\frac{i}h\Big( t (r^2-1\big)+s\big(p_2(x, r\Theta)-E_2\big)+O_{x,\xi_1}(t^2) \nonumber\\ 
&+O_{x,\xi_1}(s^2) \Big)\Big]c_3(s,t, x,r\Theta;h)[a(x, r\Theta)+o(1)] [\overline{a(x, r\Theta)}+o(1)] r^{n-1} dr d\Theta dt ds.
\end{align}

Again one would like to apply stationary phase in $(r, t)$-variables. The critical point $(r_c,t_c)$ satisfies that
\begin{align}
2r t  +s\Theta\cdot \partial_{\xi_1} p_2(x, r\Theta)+O_{x,\xi_1}(t^2)+O_{x,\xi_1}(s^2)&=0, \label{r}\\
(r^2-1)+O_{x,\xi_1}(t)&=0. \label{E_1}
\end{align}

Taking $\epsilon_0$ small enough in the definition of $\rho$ and combining \eqref{r} and \eqref{E_1}, we can get that
\begin{equation}\label{t_c}
t_c=O(s).
\end{equation}

Taking derivative about $\Theta$ in \eqref{r}, \eqref{E_1} and combining \eqref{E_1} one can derive that
\begin{equation}\label{d t_c}
\partial_{\theta_j} t_c =O(s).
\end{equation}

Now performing stationary phase at the critical point $\big(r_c, t_c \big)$ in \eqref{tilde I} and Taylor expansion gives,
\begin{align}
I (x, x; h)=&h^{-n+1}\int\int \hat{\rho}(s)  \exp\Big[\frac{i}h s\big(p_2(x, \Theta)-E_2\big)+O_{x,\Theta}(st_c) +O_{x,\Theta}(t_c^2)+O_{x,\Theta}(s^2) \Big] \nonumber\\ 
& \cdot c_4(s, x, \Theta; h) [a(x, \Theta)+o(1)] [\overline{a(x, \Theta)}+o(1)] d\Theta ds,
\end{align}
here $c_4\sim \sum_{j=0}^\infty c_{4,j} h^j$ and $c_{4,j}\in C^\infty$.

{\bf Case I}. 
If $|s|\leq h$, it's straightforward to get that
\begin{equation}
\Big| I (x, x; h) \Big|= O(h^{-n+2}).
\end{equation}

If $|s|>h$, one can perform the stationary phase theorem about $\Theta$ due to the assumption that $(p_1, p_2)$ is Morse type. Hence
\begin{equation}
\Big| I (x, x; h) \Big|\leq C |a(x,\xi_0)|^2 h^{\frac{-n+1}2} \Big|\int_{|s|>h}s^{\frac{-n-1}2} ds\Big|=O(1) |a(x,\xi_0)|^2 I_h(x)^2.
\end{equation}

{\bf Case II}. If $(p_1|_{\Gamma}, p_2|_{\Gamma})$ is {\bf rank $1$} at $x$, integration by parts in some variable $\theta_j$ is necessary to deal with integration $|{I} (x,x; h)|$. 

Now one can separate the integration
\begin{align*}
\Big| I (x, x; h) \Big|& \leq \Big|h^{-n+1}\int_{|s|\leq h} \int \hat{\rho}(s)  \exp[\frac{i}h\Phi(s,x,\Theta) ] \cdot c_5(s, x, \Theta; h)d\Theta ds \Big| \nonumber \\
+& \Big|h^{-n+1}\int_{h<|s|<1}\int \hat{\rho}(s)  \exp[\frac{i}h\Phi(s,x,\Theta) ] \cdot c_5(s, x, \Theta; h)d\Theta ds\Big|
:=\Big| {I}_1 (x; h)\Big| + \Big| {I}_2 (x; h) \Big|,
\end{align*}
here we set $\Phi(s,x,\Theta)= s\big(p_2(x, \Theta)-E_2\big)+O_{x,\Theta}(st_c) +O_{x,\Theta}(t_c^2)+O_{x,\Theta}(s^2)$ and $c_5(s, x, \Theta; h)=c_4(s, x, \Theta; h) [a(x, \Theta)+o(1)] [\overline{a(x, \Theta)}+o(1)]$.

\noindent It's straightforward to get that 
\begin{equation}\label{last1}
\Big| {I}_1 (x; h)\Big|=O(h^{-n+2}).
\end{equation}

\noindent 
For the second term, with the help of \eqref{t_c} and \eqref{d t_c} the key observation is that for some $1\leq j\leq n-1$
\begin{equation}\label{phi low}
\big| \partial_{\theta_j} \Phi(s,x,\Theta) \big|\geq|s| \big(\big|\partial_{\theta_j} p_2(x, \Theta)\big|-O(\epsilon_0)\big)>C|s|>0
\end{equation}
due to the {\bf rank $1$ condition}  and the definition of $\rho$.

Also with the help of \eqref{t_c} and \eqref{d t_c} one has that
\begin{equation}\label{phi up}
\big|\partial_{\theta_j}^\alpha \Phi(s,x,\Theta)\big|\leq C|s|, \quad\text{for a given multi-index } \alpha.
\end{equation}

Hence, one can integrate by parts to get that
\begin{align}\label{last2}
\Big| {I}_2 (x; h) \Big|&=h^{-n+2} \Big|\int_{h<|s|<1} \int \frac{1}{\partial_{\theta_j} \Phi} \big(\partial_{\theta_j} \exp[\frac{i}h  \Phi(s,x,\Theta)]\big)  \cdot c_5(s, x, \Theta; h) d\Theta ds\Big| \nonumber\\
&\leq h^{-n+2}\Big|\int_{h<|s|<1} \int \frac{1}{\partial_{\theta_j} \Phi} \exp[\frac{i}h  \Phi(s,x,\Theta)]  \cdot \partial_{\theta_j} c_5(s, x, \Theta; h) d\Theta ds\Big| \nonumber\\
&+ h^{-n+2} \Big|\int_{h<|s|<1} \int \frac{\partial^2_{\theta_j}\Phi }{|\partial_{\theta_j} \Phi|^2}   \exp[\frac{i}h  \Phi(s,x,\Theta)] c_5(s, x, \Theta; h) d\Theta ds\Big|. \nonumber
\end{align}

One more integration by parts gives us that
\begin{align}
\Big| {I}_2 (x; h) \Big|
&\leq h^{-n+3}\Big|\int_{h<|s|<1} \int \frac{1}{|\partial_{\theta_j} \Phi|^2} \big(\partial_{\theta_j} \exp[\frac{i}h  \Phi(s,x,\Theta)]\big)  \cdot \partial_{\theta_j}c_5(s, x, \Theta; h) d\Theta ds\Big| \nonumber\\
&+ h^{-n+3} \Big|\int_{h<|s|<1} \int \frac{\partial^2_{\theta_j}\Phi }{(\partial_{\theta_j} \Phi)^3}  \big(\partial_{\theta_j} \exp[\frac{i}h  \Phi(s,x,\Theta)]\big) c_5(s, x, \Theta; h) d\Theta ds\Big| \nonumber\\
&\leq h^{-n+3}\Big|\int_{h<|s|<1} \int \frac{1}{|\partial_{\theta_j} \Phi|^2} \exp[\frac{i}h  \Phi(s, x, \Theta; h)]  \cdot \partial^2_{\theta_j}c_5(s, x, \Theta; h) \, d\Theta ds\Big| \nonumber\\
&+2 h^{-n+3}\Big|\int_{h<|s|<1} \int \frac{\partial_{\theta_j}^2\Phi}{(\partial_{\theta_j} \Phi)^3}  \exp[\frac{i}h  \Phi(s, x, \Theta; h)] c_5(s, x, \Theta; h) d\Theta ds\Big| \nonumber\\
&+h^{-n+3}\Big|\int_{h<|s|<1} \int \frac{\partial_{\theta_j}^2\Phi}{(\partial_{\theta_j} \Phi)^3}  \exp[\frac{i}h  \Phi(s, x, \Theta; h)] \cdot \partial_{\theta_j}c_5(s, x, \Theta; h) d\Theta ds\Big| \nonumber\\
&+ h^{-n+3} \Big|\int_{h<|s|<1} \int \big( \frac{\partial^3_{\theta_j}\Phi }{(\partial_{\theta_j} \Phi)^3}-3 \frac{(\partial^2_{\theta_j}\Phi)^2 }{(\partial_{\theta_j} \Phi)^4} \big)  \exp[\frac{i}h  \Phi(s,x,\Theta)] c_5(s, x, \Theta; h) d\Theta ds\Big| \nonumber\\
&\leq C h^{-n+3}  \Big|\int_{h<|s|<1}\frac{1}{s^2} ds \Big|=Ch^{-n+2},
\end{align}
with the help of \eqref{phi low} and \eqref{phi up} in the second last inequality.
\end{proof}

\bigskip
Now let us take a real-valued function $\chi\in \mathcal{S}(\mathbb{R})$ satisfying
\begin{equation}\label{spt}
\chi(0)=1\quad \text{and}\quad \hat\chi(t)=0,\quad |t|\geq\epsilon,
\end{equation}
here $\epsilon$ is a small positive constant.
Since $\chi\big(h^{-1}[P_1(h)-E_1(h)]\big)u_h=u_h$, by the Cauchy-Schwarz inequality, one has that
\begin{align}\label{red1}
&\Big| A u_h (x) \Big|^2 \nonumber \\
=&\Big| \Big(A \chi \big( h^{-1}[P_1(h)-E_1(h)]\big) u_h\Big) (x) \Big|^2 \nonumber\\
=&  \Big| \int \Big(A\chi\big( h^{-1}[P_1(h)-E_1(h)]\big)\Big) (x,y) u_h(y) dy  \Big|^2  \nonumber \\
\leq &  \int \Big| A\chi\big(h^{-1}[P_1(h)-E_1(h)]\big) (x,y) \Big|^2 dy.
\end{align}

Furthermore, using the orthogonality of $\{u_j^h\}$, we can deduce that
\begin{align}\label{red2}
&\int \Big| A\chi\big(h^{-1}[P_1(h)-E_1(h)]\big) (x,y) \Big|^2 dy \nonumber\\
=& \int \sum_j \chi(h^{-1}[\lambda_j^{(1)}(h)-E_1(h)]) A u_j^h(x) \overline{u_j^h(y)} \cdot \sum_l \chi(h^{-1}[\lambda_l^{(1)}(h)-E_1(h)]) \overline{A u_l^h(x)} u_l^h(y) dy \nonumber \\
=& \sum_j \rho(h^{-1}[\lambda_j^{(1)}(h)-E_1(h)]) A u_j^h(x) \overline{A u_j^h(x)}.
\end{align}
with setting $\rho(t)=(\chi(t))^2$. 

We claim that
\begin{align}\label{upshot}
&\sum_j \rho(h^{-1}[\lambda_j^{(1)}(h)-E_1(h)]) A u_j^h(x) \overline{A u_j^h(x)} \leq 
O(1)\cdot \nonumber \\ &\sup_{\{E_2;\, (E_1,E_2)\in \mathcal{P}(T^*M)\}}\sum_j \rho(h^{-1}[\lambda_j^{(1)}(h)-E_1]) 
\rho(h^{-1}[\lambda_j^{(2)}(h)-E_2])  A u^h_j (x) \overline{A u^h_j(x)}.
\end{align}




Indeed from \cite[Section 2]{Tot} one knows that there exists a constant $C_2>0$ (independent of $j$ and $h$) such that for any $h\in(0, h_0]$, and $(\lambda_j^{(1)}(h), \lambda_j^{(2)}(h))\in \text{Spec}(P_1(h), P_2(h))$, with $|\lambda_j^{(1)}(h)-E_1|\leq C_1 h$,
\begin{equation*}
\inf_{E_2}|\lambda_j^{(2)}(h)-E_2|\leq C_2 h.
\end{equation*}
So, after taking $\epsilon$ in \eqref{spt} sufficiently small there exists a constant $C_3>0$ (independent of $j$ and $h$) such that for all $j\geq 1$ and $h\in(0, h_0]$,
\begin{equation*}
\sup_{\{E_2;\, (E_1,E_2)\in \mathcal{P}(T^*M),\, |\lambda_j^{(1)}(h)-E_1|\leq C_1 h\}}\rho\big(h^{-1}(\lambda_j^{(2)}(h) -E_2)\big)\geq C_3>0.
\end{equation*}
Since the sum on the RHS of \eqref{upshot} has non-negative terms, by restricting to $\{j;\,|\lambda_j^{(1)}(h)-E_1|\leq C_1 \}$ and (after taking $\epsilon$ small enough) using that $\rho(h^{-1}[\lambda_j^{(1)}(h)-E_1])\geq C_4>0$ for these eigenvalues, one gets that
\begin{align*}
\sum_{\{j;\, |\lambda_j^{(1)}(h)-E_1|=O(h)\}} A u^h_j(x) \overline{Au^h_j(x)} \leq 
O(1)\sup_{\{E_2;\, (E_1,E_2)\in \mathcal{P}(T^*M)\}}& \sum_j \rho(h^{-1}[\lambda_j^{(1)}(h)-E_1])\cdot \nonumber \\  
&\rho(h^{-1}[\lambda_j^{(2)}(h)-E_2])  A u^h_j (x) \overline{A u^h_j(x)},
\end{align*}
which can deduce \eqref{upshot}.

With noticing that
\begin{align}\label{sum to integral}
&\sum_j \rho(h^{-1}[\lambda_j^{(1)}(h)-E_1]) \rho \big(h^{-1}[\lambda_j^{(2)}(h)-E_2]\big) A u_j^h(x) \overline{A u_j^h(x)} \nonumber\\
=&  \frac{1}{2\pi} \int \hat\rho(t) \Big(A e^{\frac{i}h t(P_1-E_1)} \rho\big(h^{-1}[P_2(h)-E_2]\big)A^* \Big)(t,x,x) dt,    
\end{align}
and combining \eqref{red1}, \eqref{red2} and \eqref{upshot}, one can apply Lemma \ref{lwl} to get following,

\begin{cor}\label{cor}
Under the same assumption of Lemma \ref{lwl}, if $(p_1|_{\Gamma}, p_2|_{\Gamma})$ with $\Gamma:=\mathcal{C}_x\cap \supp a$ is {\bf rank $1$} at $x$, one has that 
\begin{equation}\label{main cor eq2}
|A(x,hD) u_{h}|^2=O(h^{-n+2}).
\end{equation}
Otherwise, one has the following
\begin{equation}
|A(x,hD) u_{h}|^2=O(1) I_h(x)^2,
\end{equation}
if $(p_1, p_2)$ is Morse type.
\end{cor}

Furthermore, we can prove that
\begin{lem}\label{lwl 2}
Assume $A(x, hD)\in \Psi^0(M)$ with symbol $a=\sum_{j=0}^\infty a_j h^j$. Let $\rho\in C_0^\infty(\mathbb{R})$ vanish for $|t|\geq \epsilon_0$, $\epsilon_0>0$ small enough, and satisfy $\rho(0)=1$. Let $\Gamma:=\mathcal{C}_x\cap \supp a$. Then if $(p_1|_\Gamma, p_2|_\Gamma, p_3|_\Gamma)$ is {\bf rank $2$} at $x$, one has that
\begin{equation}\label{main lemma 2 eq}
\int_{\mathbb{R}} \hat{\rho}(t) \big(A e^{\frac{i}h t(P_1-E_1)}  \rho\big(h^{-1}[P_2(h)-E_2]\big)\rho\big(h^{-1}[P_3(h)-E_3]\big)  A^*\big)(t,x,x) dt =O(h^{-n+3}).
\end{equation}
\end{lem}
\begin{proof}
With repeating the begining of the proof of Lemma \ref{lwl}, one need to show that the right hand of \eqref{main lemma 2 eq} can bound
\begin{align}
I(x,x;h):=h^{- 6n}\int &\hat{\rho}(s_2)\hat{\rho}(s_1)\hat{\rho}(t) \exp[\frac{i}h \left<x-y_1, \xi\right>] \exp[\frac{i}h \Phi(s_1, s_2,t,y_1, w_1, w_2, y_2, \xi_1, \xi_2, \xi_3)]\cdot  \nonumber \\
& \exp[\frac{i}h \left<y_2-x, \eta\right>] c_0(s_1, s_2,t,y_1,w_1,w_2,y_2,\xi_1,\xi_2,\xi_3;h) a(x, \xi) \cdot \nonumber \\
 &\qquad\qquad \overline{a(x, \eta)} dy_1 d\xi dw_1 d\xi_1 dw_2 d\xi_2 d\xi_3 dy_2 d\eta dt ds_1 ds_2,
\end{align}
here $c_0\sim \sum_{j=0}^\infty c_{0,j} h^j$ with $c_{0,j}\in C^\infty$ and $c$ is supported near the energy shell $\{p_1=E_1\}\cup\{p_2=E_2\}\cup\{p_3=E_3\}$, 
\begin{align*}
\Phi(\cdot)=&\left<y_1-w_1,\xi_1 \right>+\left<w_1-w_2, \xi_2\right>+\left<w_2-y_2, \xi_3\right> +t\big(p_1(y_1,\xi_1)-E_1\big)+s_1\big( p_2(w_1,\xi_2)-E_2 \big)\nonumber \\
&+s_2\big( p_3(w_2,\xi_3)-E_3 \big)+O_{y_1,\xi_1}(t^2)+O_{w_1,\xi_2}(s_1^2)+O_{w_2,\xi_3}(s_2^2).
\end{align*}

Now apply stationary phase in $(y_1, \xi)$-variables. The critical point equations for $(y_1, \xi)$ are
\begin{align}
\xi&=\xi_1+ t\partial_{y_1} p_1(y_1,\xi_1)+ O(t^2), \\
y_1&=x.
\end{align}

With applying Taylor expansion at the critical point $(y_{1c}, \xi_{c})$ one has that
\begin{align}
&I (x,x; h):=h^{-4n}\int \hat{\rho}(s_2)\hat{\rho}(s_1)\hat{\rho}(t)  \exp\Big[\frac{i}h\Big(\left<x-w_1, \xi_1\right>+\left<w_1-w_2,\xi_2\right>+\left<w_2-y_2,\xi_3\right> \nonumber \\
&\qquad+\left<y_2-x, \eta\right> +t(p_1(x,\xi_1)-E_1\big)+ s_1\big(p_2(w_1,\xi_2)-E_2\big) +s_2\big( p_3(w_2,\xi_3)-E_3 \big)+O_{x,\xi_1}(t^2) \nonumber\\
&\qquad\qquad + O_{w_1,\xi_2}(s_1^2)+O_{w_2,\xi_3}(s_2^2) \Big)\Big]\cdot c_1(s_1, s_2,t,x, w_1,w_2,y_2,\xi_1,\xi_2,\xi_3;h)[a(x, \xi_1)+o(1)]  \nonumber\\
&\qquad\qquad\qquad \cdot \overline{a(x, \eta)} dw_1 d\xi_1 dw_2 d\xi_2 d\xi_3 dy_2 d\eta dt ds_1 ds_2, 
\end{align}
here $c_1\sim \sum_{j=0}^\infty c_{1,j} h^j$ and $c_{1,j}\in C^\infty$.

Apply stationary phase in $(y_2, \eta)$-variables. The critical point equations for $(y_2, \eta)$ are
\begin{align}
\eta&=\xi_3, \\
y_2&=x.
\end{align}
With applying stationary phase at the critical point $(y_{2c}, \eta_{c})$ and Taylor expansion one has that
\begin{align*}
I (x,x; h)=&h^{-3n}\int \hat{\rho}(s_2)\hat{\rho}(s_1)\hat{\rho}(t)  \exp\Big[\frac{i}h\Big(\left<x-w_1, \xi_1\right>+\left<w_1-w_2,\xi_2\right>+\left<w_2-x,\xi_3\right> \nonumber \\
&\quad +t(p_1(x,\xi_1)-E_1\big)+ s_1\big(p_2(w_1,\xi_2)-E_2\big) +s_2\big( p_3(w_2,\xi_3)-E_3 \big)+O_{x,\xi_1}(t^2) \nonumber\\
&\quad\quad + O_{w_1,\xi_2}(s_1^2)+O_{w_2,\xi_3}(s_2^2) \Big)\Big]\cdot c_2(s_1, s_2,t,x, w_1,w_2,\xi_1,\xi_2,\xi_3;h) [a(x, \xi_1)+o(1)]  \nonumber\\
&\quad\qquad\quad \cdot \overline{a(x, \xi_3)} dw_1 d\xi_1 dw_2 d\xi_2 d\xi_3 dt ds_1 ds_2, 
\end{align*}
here $c_2\sim \sum_{j=0}^\infty c_{2,j} h^j$ and $c_{2,j}\in C^\infty$.

Apply stationary phase in $(w_1, \xi_2)$-variables. The critical point equations for $(w_1,\xi_2)$ are
\begin{align}
\xi_2&=\xi_1-s_1\partial_{w_1} p_2(w_1,\xi_2)+ O_{w_1,\xi_2}(s_1^2), \\
w_1&=w_2-s_1\partial_{\xi_2} p_2(w_1,\xi_2)+ O_{w_1, \xi_2}(s_1^2).
\end{align}

Hence with applying Taylor expansion one can get that
\begin{align*}
I (x,x; h)=&h^{-2n}\int \hat{\rho}(s_2)\hat{\rho}(s_1)\hat{\rho}(t)  \exp\Big[\frac{i}h\Big(\left<x-w_2, \xi_1-\xi_3\right>+t(p_1(x,\xi_1)-E_1\big)+ \nonumber \\
&  s_1\big(p_2(w_2,\xi_1)-E_2\big)+s_2\big( p_3(w_2,\xi_3)-E_3 \big)+O_{x,\xi_1}(t^2)+ O_{w_2,\xi_1}(s_1^2)+O_{w_2,\xi_3}(s_2^2) \Big)\Big] \nonumber\\
&\qquad \cdot c_3(s_1,s_2,t,x,w_2,\xi_1, \xi_3;h)[a(x, \xi_1)+o(1)] \overline{a(x, \xi_3)} d\xi_1 dw_2 d\xi_3 dt ds_1 ds_2,
\end{align*}
here $c_3\sim \sum_{j=0}^\infty c_{3,j} h^j$ and $c_{3,j}\in C^\infty$.

Next apply stationary phase in $(w_2, \xi_3)$-variables. The critical point equations for $(w_2,\xi_3)$ are
\begin{align}
\xi_3&=\xi_1-s_1\partial_{w_2} p_2(w_2,\xi_1)-s_2\partial_{w_2} p_3(w_2,\xi_3)+O_{w_2,\xi_1}(s_1^2)+ O_{w_2,\xi_3}(s_2^2),\\
w_2&=x-s_2\partial_{\xi_3} p_3(w_2,\xi_3)+ O_{w_2,\xi_3}(s_2^2).
\end{align}

So applying Taylor expansion one can get that
\begin{align}
I (x,x; h)=&h^{-n}\int \hat{\rho}(s_2)\hat{\rho}(s_1)\hat{\rho}(t)  \exp\Big[\frac{i}h\Big(t(p_1(x,\xi_1)-E_1\big)+ s_1\big(p_2(x,\xi_1)-E_2\big)+ \nonumber \\
&\qquad  s_2\big( p_3(x,\xi_1)-E_3 \big)+O_{x,\xi_1}(t^2)+O_{x,\xi_1}(s_1 s_2)+ O_{x,\xi_1}(s_1^2)+O_{x,\xi_1}(s_2^2) \Big)\Big] \nonumber\\
&\qquad\qquad \cdot c_4(s_1,s_2,t,x,\xi_1;h) [a(x, \xi_1)+o(1)] [\overline{a(x, \xi_1)}+o(1)] d\xi_1 dt ds_1 ds_2,
\end{align}
here $c_4\sim \sum_{j=0}^\infty c_{4,j} h^j$ and $c_{4,j}\in C^\infty$.

Now use the geodesic normal coordinate about $x$ and make the change of variables $\xi_1=r\Theta$, where $r=|\xi_1|$ and $\Theta=(\theta_1,\theta_2,\dots,\theta_{n-1})\in S^{n-1}$. One can get
\begin{align}\label{tilde I'}
I (x, x; h)=& h^{-n}\int \hat{\rho}(s_2)\hat{\rho}(s_1)\hat{\rho}(t)  \exp\Big[\frac{i}h\Big( t (r^2-1\big)+s_1\big(p_2(x, r\Theta)-E_2\big)+s_2\big( p_3(x,r\Theta)-E_3 \big) \nonumber\\ 
&+O_{x,\xi_1}(t^2)+O_{x,\xi_1}(s_1 s_2)+ O_{x,\xi_1}(s_1^2)+O_{x,\xi_1}(s_2^2) \Big)\Big]c_4(s_1,s_2,t, x,r\Theta;h)\nonumber \\
&\quad\cdot[a(x, r\Theta)+o(1)] [\overline{a(x, r\Theta)}+o(1)] r^{n-1} dr d\Theta dt ds_1 ds_2.
\end{align}

Again one would like to apply stationary phase in $(r, t)$-variables. The critical point $(r_c,t_c)$ satisfies that
\begin{align}
2r t  +s_1\Theta\cdot \partial_{\xi_1} p_2(x, r\Theta)+s_2\Theta\cdot \partial_{\xi_1} p_3(x, r\Theta)+O_{x,\xi_1}(t^2)+O_{x,\xi_1}(s_1 s_2)\qquad\qquad & \nonumber \\
+ O_{x,\xi_1}(s_1^2)+O_{x,\xi_1}(s_2^2)&=0, \label{2r}\\
(r^2-1)+O_{x,\xi_1}(t)&=0. \label{2E_1}
\end{align}

Taking $\epsilon_0$ small enough in the definition of $\rho$ and combining \eqref{2r} and \eqref{2E_1}, we can get that
\begin{equation}\label{t_c 2}
t_c=O(s_1)+O(s_2).
\end{equation}

Taking derivative about $\theta_j$ in \eqref{2r}, \eqref{2E_1} and combining \eqref{2E_1} one can derive that
\begin{equation}\label{d t_c 2}
\partial_{\theta_j} t_c =O(s_1)+O(s_2).
\end{equation}

Now performing stationary phase at the critical point $\big(r_c, t_c \big)$ in \eqref{tilde I'} and Taylor expansion gives,
\begin{align}
I (x, x; h)=&h^{-n+1}\int \hat{\rho}(s_1)\hat{\rho}(s_2)  \exp\Big[\frac{i}h \Big(s_1\big(p_2(x, \Theta)-E_2\big)+s_2\big(p_3(x, \Theta)-E_3\big)+O_{x,\Theta}(s_1t_c) \nonumber\\ 
& +O_{x,\Theta}(s_2t_c)+O_{x,\Theta}(t_c^2)+O_{x,\Theta}(s_1 s_2)+O_{x,\Theta}(s_1^2)+O_{x,\Theta}(s_2^2)\Big) \Big] c_5(s_1, s_2, x, \Theta; h)\nonumber\\
&\cdot [a(x, \Theta)+o(1)] [\overline{a(x, \Theta)}+o(1)] d\Theta ds_1 ds_2,
\end{align}
here $c_5\sim \sum_{j=0}^\infty c_{5,j} h^j$ and $c_{5,j}\in C^\infty$.

Then one can separate the integration
\begin{align*}
\Big| I (x, x; h) \Big|& \leq \Big|h^{-n+1}\int_{|s_1|,|s_2|\leq h} \int \hat{\rho}(s_1)\hat{\rho}(s_2)  \exp[\frac{i}h\Phi(s_1,s_2,z,\Theta) ] \cdot c_6(s_1, s_2, x, \Theta; h)d\Theta ds_1 ds_2 \Big| \nonumber \\
+&\Big|h^{-n+1}\int_{|s_1|\leq h,h<|s_2|<1}\int \cdot\, d\Theta ds_1ds_2\Big|+\Big|h^{-n+1}\int_{h<|s_1|<1,|s_2|\leq h}\int \cdot\, d\Theta ds_1ds_2\Big| \nonumber\\
+& \Big|h^{-n+1}\int_{h<|s_1|,|s_2|<1}\int \hat{\rho}(s_1)\hat{\rho}(s_2)  \exp[\frac{i}h\Phi(s_1,s_2,z,\Theta) ] \cdot c_6(s_1,s_2, x, \Theta; h)d\Theta ds_1ds_2\Big| \nonumber\\
&:=\Big| {I}_1 (x; h)\Big| + \Big| {I}_2 (x; h) \Big|+ \Big| {I}_3 (x; h) \Big|+ \Big| {I}_4 (x; h) \Big|,
\end{align*}
here we set $\Phi(s_1,s_2,x,\Theta)= s_1\big(p_2(x, \Theta)-E_2\big)+s_2\big(p_3(x, \Theta)-E_3\big)+O_{x,\Theta}(s_1t_c) +O_{x,\Theta}(s_2t_c)+O_{x,\Theta}(t_c^2)+O_{x,\Theta}(s_1 s_2)+O_{x,\Theta}(s_1^2)+O_{x,\Theta}(s_2^2)$ and $c_6(s_1,s_2, x, \Theta; h)=c_5(s_1, s_2, x, \Theta; h) [a(x, \Theta)+o(1)] [\overline{a(x, \Theta)}+o(1)]$.

We only focus on $\Big| {I}_1 (x; h) \Big|$ and $\Big| {I}_4 (x; h) \Big|$ since the rest can be dealt with similarly. It's straightforward to get that 
\begin{equation}\label{last1'}
\Big| {I}_1 (z; h)\Big|=O(h^{-n+3}).
\end{equation}

For $\Big| {I}_4 (x; h) \Big|$, 
with the help of \eqref{t_c 2} and \eqref{d t_c 2} the key observation is that for some $1\leq j\neq k\leq n-1$
\begin{align}
\big| \partial_{\theta_j} \Phi(s_1,s_2,x,\Theta) \big|&\geq|s_1| \big(\big|\partial_{\theta_j} p_2(x, \Theta)\big|-O(\epsilon_0)\big)>C_1|s_1|>0, \label{rank2 1}\\
\big| \partial_{\theta_k} \Phi(s_1,s_2,x,\Theta) \big|&\geq|s_2| \big(\big|\partial_{\theta_k} p_3(x, \Theta)\big|-O(\epsilon_0)\big)>C_2|s_2|>0 \label{rank2 2}
\end{align}
due to the {\bf rank $2$ condition}  and the definition of $\rho$.

Also with the help of \eqref{t_c 2} and \eqref{d t_c 2} one has that
\begin{equation}\label{s1s2}
\big|\partial_{\Theta}^\alpha \Phi(s_1, s_2, x,\Theta)\big|=O(s_1)+O(s_2), \quad\text{for a given multi-index } \alpha.
\end{equation}

Hence, one can adapt the argument from the end of Lemma \ref{lwl}: integrate by parts twice with respect to variables $\theta_j$ and $\theta_k$, then apply equation \eqref{s1s2} to conclude that
\begin{align}\label{last2'}
\Big| {I}_4 (x; h) \Big| \leq C h^{-n+5}  \Big|\int_{h<|s_1|,|s_2|<1}\frac{1}{s^2_1 s^2_2}  ds_1 d_2 \Big|=O(h^{-n+3}) .
\end{align}

\end{proof}

Following a similar argument after Lemma \ref{lwl} and combining Lemma \ref{lwl 2}, one can inductively obtain Theorem \ref{trace}.

\bigskip

\section{Proof of Theorem \ref{main1}}\label{section3}
The proof of \eqref{main1 eq2} is an application of Theorem \ref{trace}. Hence one only need to focus on \eqref{main1 eq} and the proof is mainly inspired by \cite{STZ11} with help of the results in Section \ref{section2}. We first introduce some notation. One sets $L_\delta(x,\xi)$, if it exists, is the length of shortest loop such that $\text{dist}(G^{\ell}_x(\xi),\xi)\leq \delta$, with $\text{dist}(\cdot, \cdot)$ being the distance induced by the metric, $\xi$ is the initial direction and $G^{\ell}_x(\xi)$ is the terminal direction of all loop with a given length $\ell$. Let $\mathcal{R}_x^\delta$ be all the $\xi\in \bigcap\limits_{2\leq j \leq n}\{|p_j(x,\xi)-E_j|\leq \epsilon\}$ satisfying $L_\delta(x,\xi)$ is finite.

Without generality, we work on a coordinate patch of a point $x$, which we identified with an open subset of $\mathbb{R}^n$. By continuity, one can choose an open set $U$ such that $\Gamma\subsetneqq U$ and $\mathcal{P}|_{U}$ is rank $k$ ($k\geq 2$ if $n>3$) at $x$, with a reminder that $\Gamma=\mathcal{R}_x\cap \bigcap\limits_{2\leq j \leq n}\{|p_j(x,\xi)-E_j|\leq \epsilon\}$. Then we fix a small number $\delta>0$ so that 
\begin{equation}
\Gamma\subset \mathcal{R}_x^\delta\subset U.
\end{equation}

For given a large number $T$ (to be specified later), applying $C^\infty$ Urysohn lemma, there exists a function $b\in C^\infty(S^{n-1})$ with range $[0,1]$ such that 
\begin{equation}
\mathcal{R}_x^\delta\subset \supp{b}\subset U,
\end{equation}
and
\begin{equation}
    L_\delta(x,\xi)\geq 2T \quad \text{ for } \forall \xi \in \supp B,
\end{equation}
where $B(\xi)=1-b(\xi)$. 

From Theorem \ref{trace} one has that
\begin{equation}
|b(x, hD)u_h(x)|=o(1)I_n(h).
\end{equation}

So one only need to deal with $|B(x, hD) u_h(x)|$. We construct a smooth partition of unity $1=\sum_k \psi_k(\xi)$ of $S^{n-1}$ which consists of $O(\delta^{-(n-1)})$ terms such that each term has range in $[0,2]$ and is supported in a spherical cap of diameter smaller than $\delta/10$. Let $B_k(x,hD)$ be the zero-order h-pseudo-differential operator whose symbol is $B(x,\xi)\psi_k(\xi)$. Hence
\begin{equation}\label{terms}
    |B(x, hD) u_h(x)| \leq \sum_k \Big| B_k u_h (x) \Big|.
\end{equation}

By Cauchy-Schwarz inequality and using the orthogonality of $\{u_j^h\}$, one has that
\begin{align}
&\Big| B_k u_h (x) \Big|^2 \nonumber \\
=&\Big| \big(B_k \chi \big(T h^{-1}[P_1(h)-E_1]\big) u_h\big) (x) \Big|^2 \nonumber\\
=&  \Big| \int \Big(B_k\chi\big(T h^{-1}[P_1(h)-E_1]\big)\Big) (x,y) u_h(y) dy  \Big|^2  \nonumber \\
\leq &  \int \Big| B_k\chi\big(Th^{-1}[P_1(h)-E_1]\big) (x,y) \Big|^2 dy \, \|u_h(y)\|^2_{L^2(M)} \nonumber\\
=& \int \sum_j \chi(Th^{-1}[\lambda_j^{(1)}(h)-E_1]) B_ku_j^h(x) \overline{u_j^h(y)} \sum_l \chi(Th^{-1}[\lambda_l^{(1)}(h)-E_1]) \overline{B_k u_l^h(x)} u_l^h(y) dy \nonumber\\
=& \sum_j \rho(Th^{-1}[\lambda_j^{(1)}(h)-E_1]) B_k u_j^h(x) \overline{B_k u_j^h(x)} \nonumber\\
=& T^{-1} \int \hat\rho(t/T) (B_k e^{\frac{i}h t(P_1-E_1)} B_k^*)(t,x,x) dt,
\end{align}
with setting $\rho(t)=(\chi(t))^2$.

Now one can insert a new cut-off function $\hat\beta(t)\in C_0^\infty(\mathbb{R})$ satisfying $\hat\beta(t)=1$, $|t|<1$ and $\hat\beta(t)=0$, $|t|>2$ to get that
\begin{align}
\Big| B_k u_h (x) \Big|^2 &\leq  T^{-1} \int \big(1-\hat\beta(t)\big)\hat\rho(t/T) (B_k e^{\frac{i}h t(P_1-E_1)}B_k^*)(t,x,x) dt \nonumber\\
&\qquad\qquad + T^{-1} \int \hat\beta(t)\hat\rho(t/T) (B_k e^{\frac{i}h t(P_1-E_1)}B_k^*)(t,x,x) dt \nonumber\\
&=:I(x,x)+II(x,x).
\end{align}

Following \cite{STZ11} one has that
\begin{equation}\label{BI}
I(x,x) =O(h^\infty).
\end{equation}

For the second part, with setting $\hat\alpha(t)=\hat\beta(t)\hat\rho(t/T)$, one notices that
\begin{align*}
II(x,x)&=T^{-1} \int \hat\beta(t)\hat\rho(t/T) (B_ke^{\frac{i}h t(P_1-E_1)} B_k^*)(t,x,x) dt \\
&=T^{-1} \sum_j \alpha(h^{-1}[\lambda_j^{(1)}(h)-E_1]) B_k u_j^h(x) \overline{B_k u_j^h(x)}.
\end{align*}

Following the same argument of \eqref{upshot} and \eqref{sum to integral}, one can apply Lemma \ref{lwl} to get that
\begin{equation}\label{BII}
II(x,x) \leq C T^{-1} I_h(x)^2.
\end{equation}
Combining \eqref{terms}, \eqref{BI} and \eqref{BII} with  noticing that there are $O(\delta^{-(n-1)})$ terms, one can take $T$ sufficient large to get that
\begin{equation}
    |B(x, hD)u_h(x)|=o(1)I_n(h).
\end{equation}

\bigskip

\section{Examples}\label{section4}

\subsection{Surface of revolutions}
We study the surface of revolution as an example to demonstrate several points mentioned in the Introduction. Proposition \ref{SOR degen} shows the geometric meaning of the rank $1$ condition in this simple case; Proposition \ref{SOR recur} and \ref{SOR quasi} clarifies the point made in Remark (1) of Theorem \ref{main1}. We believe that some constructions can be generalized to Riemannian warped products.

In general, a surface of revolution embedded in $\R^3$ is given by parametrization
\begin{equation}\label{SoR parametrization}
    \begin{cases}
        x=f(t)\cos\phi\\
        y=f(t)\sin\phi\\
        z=g(t),
    \end{cases}
\end{equation}
for $t\in [t_-,t_+]$ and $\phi\in [0,2\pi]$. Its first fundamental form is 
\begin{align*}
    E&=e_t\cdot e_t=f'(t)^2+g'(t)^2\\
    F&=e_t\cdot e_\phi=0\\
    G&=e_\phi\cdot e_\phi=f^2(t)
\end{align*}
and we want to arrange for $f'(t)^2+g'(t)^2=1$, so that the metric is $ds^2=dt^2+f^2(t)d\phi^2$. We call this metric $g$, not to be confused with the above function $g$. One can compute the Laplacian in the coordinate system $(t,\phi)$,
\begin{align}
    \Delta &=\frac{1}{\sqrt{|g|}}\sum\partial_j(\sqrt{|g|}\,g^{jk}\partial_k)\nonumber\\
    &=\partial_t^2+f^{-2}(t)\partial_\phi^2+\frac{f'(t)}{f(t)}\partial_t.
\end{align}
The principal symbol is then $|\xi|^2_g=\xi_t^2+f^{-2}(t)\xi_{\phi}^2$. The quantum integrable system is defined by $\mathcal{P}=(p_1,p_2)$, where $p_1=|\xi|^2_g$, $p_2=\xi_{\phi}$.

We will make extra assumptions in order to make the problem simpler. We choose $f$ such that $f(t_-)=f(t_+)=0$ and assume that $f$ makes the metric smooth at the poles. We also assume that $f''(t_0)\neq 0$ for the $t_0$ where $f'(t_0)=0$. The assumption guarantees that the set $\{(t,\phi)\in M:f'(t)=0\}$ consists of disjoint union of horizontal circles. For simplicity, we may assume there is only one such circle, which we refer to as the \textbf{equator}.

There is a simple characterization of the rank $1$ condition in this case.

\begin{prop}\label{SOR degen}
    Let $M$ be a surface of revolution as above, and $x\in M$ not a pole nor on the equator. Then $\mathcal{P}$ satisfies the rank $1$ condition on $\Gamma=\{(\xi_t,\xi_{\phi})\in S^*_xM: \xi_t\neq 0\}$ at $x$.
\end{prop}

\begin{proof}
    At $x=(t,\phi)$, the coshpere is an ellipse $\mathcal{C}_x=\{\xi:\xi_t^2+\frac{\xi_{\phi}^2}{f^2(t)}=1\}$. Then, ${p_2}|_{\mathcal{C}_x}=\xi_{\phi}|_{\mathcal{C}_x}$ is critical only if $\xi_{t}=0$, that is, when the covector is vertical. Then, $\mathcal{P}$ satisfies the rank $1$ condition on $\{\xi_t\neq 0\}$ at $x$.  
\end{proof}

\begin{rem}
    Note that the trick in Section \ref{section superint} cannot be used here.
\end{rem}

The next proposition characterizes the recurrent set of certain surfaces of revolution.

\begin{prop}\label{SOR recur}
    Let $M$ be a surface of revolution as above with a real-analytic profile function $f$, and that $g$ is not a Zoll metric. Let $x\in M$ be a point on the equator, then the recurrent set $\mathcal{R}_x$ is a dense subset in $S^*_xM$ and $\mu(\mathcal{R}_x)=0$, where $\mu$ is the induced Liouville measure on $S^*_xM$.
\end{prop}

\begin{rem}
    In this case, $\mathcal{R}_x=\widetilde{\mathcal{L}}_x=\mathcal{L}_x$. Also, the conclusion holds actually for any $x$ that is not a pole. To keep the proof short, we assume that it lies on the equator. 
    
    We can almost produce a one-line proof following \cite[Thm 3.5.24]{Kli95}, which claims that the geodesic flow on a surface of revolution is "equivalent" to certain linear torus action, the slope of which is explicit. However, we opt for an elementary proof as below.
\end{rem}

\begin{figure}[h]
\includegraphics[width=\textwidth]{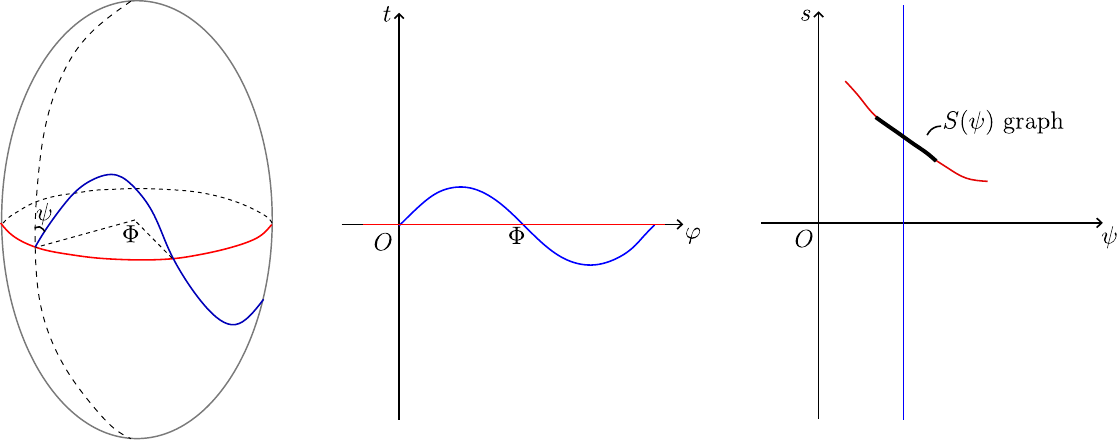}
\caption{The relation of the equator (red) and the geodesic with initial angle $\psi$ (blue) in the proof. 
Left: trajectories on the surface. Middle: canonical coordinates $(t,\phi)$. Right:  geodesic normal coordinates $(s,\psi)$.}
\label{Figure 2}
\end{figure}

\begin{proof}
We invoke the Clairaut relation that holds on every surface of revolution: For a geodesic $l$, the quantity $\gamma=f(t)\sin(\psi(t))$ is a constant along $l$, where $\psi(t)$ is the angle between the geodesic $l$ and the meridian it intersects with at $(t,\phi)$. Write $\psi=\psi(t_0)$ for the initial angle. It is known that the geodesic oscillates within a tubular region near the equator (\cite{Kli95}).

Fix $x$ on the equator. Without loss of generality, let $\phi(x)=0$. For a geodesic $l$ starting at $x$ with angle $\psi$ to the meridian, it is parametrized by arc length so that $l(s)=(t(s),\phi(s))$, $t'(s)^2+f'(t(s))^2\phi'(s)^2=1$. It satisfies the geodesic equations
\begin{equation}
    \begin{cases}
        \frac{d^2t}{ds^2}=f(t(s))\,\frac{df}{ds}\left(\frac{d\phi}{ds}\right)^2\\
        \frac{d}{ds}\left(f^2(t(s))\frac{d\phi}{ds}\right)=0
    \end{cases}
\end{equation}
subject to the initial conditions
\begin{align}
    \phi(0)=0,\quad t(0)=0,\quad \frac{d\phi}{ds}(0)=\frac{\sin\psi}{f(t_0)},\quad \frac{dt}{ds}(0)=\cos\psi.
\end{align}

The only free variable in the initial conditions is $\psi$. Using the fundamental theorem of ODE, one can write a real-analytic function $L:(\psi,s)\mapsto(t(\psi,s),\phi(\psi,s))=(L_1,L_2)$, which serve as a $C^{\omega}$ change-of-coordinates into the geodesic normal coordinates $(\psi,s)$. Then, the equator as a curve $\{(t,\phi):t=0\}$ can be expressed as $\{(\psi,s):s=S(\psi)\}$, a graph of some real-analytic function $S$. Therefore, the function $\Phi(\psi)=L_2(\psi,S(\psi))$ is real-analytic, possibly multi-valued function, which denotes the change of longitude when the geodesic $l$ returns to the equator. We take the smallest positive value as a well-defined branch, and still denote by $\Phi$. Its inverse $\Phi^{-1}$ is also real-analytic in a suitable domain.

The geodesic is recurrent if and only if $\Phi(\psi)$ is a rational multiple of $2\pi$. Let $B:=\{\text{rational multiples of }2\pi\}$. Suppose $\exists A\subset S^1,\, \mu(A)>0$ such that $\Phi(A)= B$. Then $\exists A'\subset A$, $\mu(A')>0$, $\Phi(A')=\{c\}$. By analyticity, $\Phi=const.$ Since $M$ possess rotational symmetry, we get the same constant for every geodesic. It forces the manifold to be a Zoll surface.

By assumtion, the surface of revolution is not Zoll. Then, the inverse image $\Phi^{-1}(B)$ must have zero measure. Denseness follows from the fact that $\Phi^{-1}$ is $C^{\omega}$ and that $B$ is dense.

\end{proof}

We mention a method for constructing a sequence of joint quasimodes from a sequence of eigenfunctions of its horizontal factor. We believe that such a construction also exists in certain Riemannian warped products.

\begin{prop}\label{SOR quasi}
    Let $M$ be a surface of revolution as above with profile function $f$ such that $f\leq 1$, and $\frac{|f'|}{|1-f^2|}$ is bounded where $f=1$. Then there exists an $O(h)$-quasimode $u_h$ such that $|u_h(t_0,\phi)|\simeq h^{-1/4}$ and $|u_h(t,\phi)|\simeq  e^{-C(t,\phi)/h}$ for $f(t)\neq 1$. 
\end{prop}
\begin{proof}
    Separation of variables: Let $u_{\lambda}(t,\phi)=T(t)e^{i\lambda\phi}$. We call the following \textbf{highest-weight quasimode}: 
\begin{equation}\label{trial}
    u_{\lambda}(t,\phi)=\lambda^{1/4}\exp\{\lambda\int\sqrt{\frac{1-f^2}{f^2}}dt\}\,e^{i\lambda \phi}.
\end{equation}
One can check by direct computation that 
\begin{align}
\Delta u_{\lambda}&=\Big(\partial_t^2+f^{-2}(t)\partial_\phi^2+\frac{f'(t)}{f(t)}\partial_t\Big)u_{\lambda} \nonumber\\
&=\Big(-\lambda^2-\lambda\frac{f'(t)}{f(t)^2\sqrt{1-f(t)^2}}+\lambda\frac{f'(t)\sqrt{1-f(t)^2}}{f(t)^2}\Big)u_{\lambda}\nonumber\\
&=\Big(-\lambda^2-\frac{\lambda f'(t)}{\sqrt{1-f(t)^2}}\Big)u_{\lambda}
\end{align}
At points where $f(t_0)=1$, we expect that on the equator $f'(t_0)=0$ at correct order s.t. the second term does not blow up. It is guaranteed by our assumption.

\end{proof}

\begin{rem}
Back to the round sphere, we let $f=\cos t$, then 
$$
\exp\{\lambda\int\sqrt{\frac{|1-f^2|}{f^2}}dt\}=\exp\{\lambda\int |\tan t|dt\}=\exp\{\lambda \ln|\cos t|\}=(\cos t)^{\lambda}. 
$$
(in this convention, $t\in [-\pi/2,\pi/2]$, so $\cos t>0$.)
\end{rem}

Combining the Proposition \ref{SOR recur} and \ref{SOR quasi}, we exemplified Remark (1) of Theorem \ref{main1}, and partly justified the necessity of the rank $k$ condition of nondegeneracy.

Finally, we mention that most results in Section \ref{section4} can be generalized to Riemannian warped products $I \times_fN$. It seems interesting to explore their utilities as (counter-)examples in higher dimensions.

The following examples are inspired by \cite{GT2020}.

\subsection{Laplacians on Liouville tori}

Consider the two-torus $M=\mathbb{R}^2/\mathbb{Z}^2$ with a Liouville metirc given by $g=(a(x_1)+b(x_2))(dx_1^2+dx_2^2)$, here $a, b:\mathbb{R}/\mathbb{Z}\to \mathbb{R}^+$ are two, smooth, positive periodic functions with $\min_{0\leq x_1\leq1} a(x_1)>\max_{0\leq x_2 \leq 1}b(x_2)$. The associated Laplacian $$P_1(h)=-[a(x_1)+b(x_2)]^{-1}((h\partial_{x_1})^2+(h\partial_{x_2})^2)$$ is QCI with $$P_2(h)=-[a(x_1)+b(x_2)]^{-1}(b(x_2)(h\partial_{x_1})^2-a(x_1)(h\partial_{x_2})^2).$$

Taking any point $z_0=(x_0, y_0)\in\mathbb{R}^2/\mathbb{Z}^2$ with $\alpha=a(x_0)>b(y_0)=\beta$ we have that $$p_2\vline_{\,T_{z_0}^*}=\beta(\alpha+\beta)^{-1}\xi^2- \alpha(\alpha+\beta)^{-1}\eta^2.$$
Since $S_{z_0}^*=\{(\xi,\eta);\, \xi^2+\eta^2=\alpha+\beta>0\}$, one sets $\xi=\sqrt{\alpha+\beta}\cos\theta$ and $\eta=\sqrt{\alpha+\beta}\sin\theta$. Hence $$p_2\vline_{\,S_{z_0}^*}(\theta)=\beta\cos^2\theta-\alpha\sin^2\theta,$$
which is critical at $\theta=\theta_k=k\pi/2$, $k=0,1,2,3$. One can find the corresponding energy $E_2^k:=p_2(\theta_k)$. Hence {\em rank 1 condition} is valid if $(1, E_2)\in \mathcal{B}\backslash\{(1,E_2^k)\}$, here $\mathcal{B}$ is the set of regular values of the moment map $\mathcal{P}$.

\subsection{Laplacians on ellipsoids}
Consider the ellipsoid $\mathcal{E}=\{x\in \mathbb{R}^3; \sum_{j=1}^3\frac{x^2_j}{a^2_j}=1\}$ where ${0<a_3<a_2<a_1}$ are fixed constants. It is QCI with reduced Hamiltonian $H$ and the reduced integral in involution $P$ which are given by the formulas \cite{Tot97}:
\begin{align*}
H&=a_3(x_1\xi_2-\xi_1x_2)^2+a_2(x_1\xi_3-x_3\xi_1)^2+a_1(x_3\xi_2-x_2\xi_3)^2,\\
P&=(x_1\xi_2-\xi_1x_2)^2+(x_1\xi_3-x_3\xi_1)^2+(x_3\xi_2-x_2\xi_3)^2.
\end{align*}
Here $P$ is equal to the principle symbol of the standard Laplacian on $S^2$.
Firstly, we consider the point  $(x_0, \xi_0)$  on the hyperbolic geodesics $\Gamma^{\pm}=\{(x,\xi)\in S^*(S^2); x_2=\xi_2=0\}$. One has that
\begin{align*}
H\vline_{\,\Gamma^\pm}&=a_2(x_1\xi_3-x_3\xi_1)^2,\\
P\vline_{\,\Gamma^\pm}&=(x_1\xi_3-x_3\xi_1)^2=\frac{1}{a_2}  H |_{\,\Gamma^\pm}.
\end{align*}
Hence {\em rank k condition} fails. As \cite{Tot97} points out that one can find a sequence of eigenfunctions $\{u_h\}$ such that $\|u_h\|_{L^\infty(\Gamma(h))}=O(h^{1/2}/|\ln h|)$, here $\Gamma(h)$ is a tube of width $O(h)$ about a hyperbolic geodesic $\Gamma$.

Now we consider the point $x_0$ away $\pi(\Gamma)$ where $\pi$ is the standard projection map. Denote the rectangles $R_+=(0, T_1)\times (0, T_2)$ and $R_-=(T_1, 2T_2)\times(0, T_2)$. We let $\Phi_\pm: R_\pm\to \mathcal{E}\cap \{w_2\neq 0\}$ be the conformal mapping sending vertices of $R_\pm$ to the four umbilic points $p_k; k=1,\dots 4$ of $\mathcal{E}$. We choose orientations so that $\Phi_+(x,T_2)=\Phi_-(2T_1-x,T_2)$ and $\Phi_+(x,0)=\Phi_-(2T_1-x,0)$. We henceforth let $\Phi:=\Phi_\pm: R\to \mathcal{E}$ denote the induced conformal mapping with $\Phi\vline_{R_\pm}=\Phi_{\pm}$ and $R:=R_+\cup R_-$. One has that the intrinsic Riemannian metric $\mathcal{E}$ pulled-back to $R$ is locally of Liouville form \cite{GT2020} $$ds^2=\big(a(x_1)+b(x_2)\big)(dx_1^2+dx_2^2).$$

Then the same argument works as in the Liouville torus. For the specific joint eigenvalues $\{(1, E_2^k)\}$ whose corresponding geodesics including ones that connect $x_0$ and the four umbilic points and have a focusing effect which are forward and backwards asymptotic to the middle-length ellipse \cite{Tot97, GT2020}, the {\em rank k condition fails}. However {\em rank 1 condition} is valid if $(1, E_2)\in \mathcal{B}\backslash\{(1,E_2^k)\}$ which is consistant with the result of \cite{Tot97} which shows that there exist a sequence of eigenfunctions $\{u_h\}$ such that $\|u_h\|_{L^\infty}=O(1/|\ln h|)$ outside an arbitrarily small (but fixed) neighbourhood of $\Gamma$.

\bigskip
\bibliography{reference}
\bibliographystyle{alpha}

\end{document}